\documentclass{amsart}

\title{An improved algorithm for checking the Collatz Conjecture for all $n < 2^N$}
\author{Vigleik Angeltveit}
\address{Mathematical Sciences Institute \\
Australian National University \\
Canberra, ACT 0200 \\
Australia}

\usepackage{amsxtra}
\usepackage{amsfonts}
\usepackage[latin1]{inputenc}
\usepackage{graphicx}
\usepackage{subcaption}
\usepackage{amsmath,amssymb,latexsym,amsthm,mathrsfs}
\usepackage[all]{xy}
\usepackage{pstricks}
\usepackage{verbatim}

\usepackage[ruled,vlined]{algorithm2e}

\newtheorem{theorem}{Theorem}[section]
\newtheorem{thm}[theorem]{Theorem}
\newtheorem{lemma}[theorem]{Lemma}

\newtheorem{cor}[theorem]{Corollary}

\theoremstyle{definition}
\newtheorem{defn}[theorem]{Definition}

\makeatletter
\let\c@equation\c@theorem
\makeatother
\numberwithin{equation}{section}

         \newcommand{\bZ}{\mathbb{Z}}

\DeclareMathOperator{\dip}{dip}

\numberwithin{figure}{section}

\begin{document}

\begin{abstract}
We describe a new algorithm for verifying the Collatz conjecture for all $n < 2^N$ for some fixed $N$. The algorithm takes less than twice as long to verify convergence for all $n < 2^{N+1}$ as it does to verify convergence for all $n < 2^N$.

We also discuss verification of the analogue of the Collatz conjecture for negative numbers.
\end{abstract}

\maketitle

\section{Introduction}
Let $n \geq 1$ be a natural number, and define the Collatz function $C(n)$ and the modified Collatz function $T(n)$ by
\[
 C(n) = \begin{cases} n/2    \quad & \textnormal{if $n$ is even} \\
                      3n + 1 \quad & \textnormal{if $n$ is odd}
        \end{cases}
        \qquad
 T(n) = \begin{cases} n/2        \quad & \textnormal{if $n$ is even} \\
                      (3n + 1)/2 \quad & \textnormal{if $n$ is odd}
        \end{cases}
\]
The Collatz conjecture states that for any $n$, by applying $C$ enough times we eventually reach the orbit $1 \mapsto 4 \mapsto 2 \mapsto 1$. Equivalently, by applying $T$ enough times we eventually reach the orbit $1 \mapsto 2 \mapsto 1$.

A proof of the Collatz conjecture appears to be out of reach with current techniques, but considerable computer resources have been devoted to verifying the Collatz conjecture for all $n < K$ for increasingly large $K$. The most recent progress on the subject is described in \cite{Ba25}, where Barina verifies convergence for all $n < 2^{71}$, and we will refer to this paper repeatedly for the current state of the art.

The goal of the current paper is to describe a more efficient algorithm for verifying the Collatz conjecture up to some $K$. Unlike previous algorithms, the algorithm presented in this paper becomes more efficient for higher $K$ in the sense that the fraction of integers in $[1, 2^{N+1}]$ that must be checked is smaller than the fraction of integers in $[1, 2^N]$ that must be checked. Indeed, as $N \to \infty$ the fraction of integers in $[1, 2^N]$ that must be checked approaches $0$. The number of integers that must be checked does of course still go to infinity. For the rest of the paper we fix $N$, and explain how to check that any $n < 2^N$ converges.

The key idea is to consider all $n$ with the same $k$ least significant bits (when written in binary) at the same time, and then do a recursive search where we add one bit at a time.

As a further optimisation, we precompute some bitvectors of length $2^B$ and use one of them to add the $A$ most significant bits in one go by looking $B$ steps ahead, for some chosen constants $A$ and $B$ with $B \geq A$. This makes the tail end of the recursive search faster at the same time as it excludes more cases.

At a high level this can be described as in Algorithm \ref{a:highlevel}.

\begin{algorithm}[H] \label{a:highlevel}
\begin{enumerate}
\item By recursively adding more bits, starting with the least significant bit, exclude all $n_0 < 2^{N-A}$ for which $T^{k}(n_0) < n_0$ for some $0 < k \leq N - A$, and at least some $n_0 < 2^{N-A}$ for which there exists an $n_1 < n_0$ with $T^k(n_0) = T^{\ell}(n_1)$. Simultaneously compute $T^{N-A}(n_0)$ for the remaining possible $n_0$. 

\item Use precomputed bitvectors to exclude all $n = n_0 + a 2^{N-A} < 2^N$ for which $T^{k}(n) < n$ for some $N - A < k \leq N - A + B$, all $n \equiv 2, 4, 5, 8 \mod 9$, and at least some $n < 2^N$ for which there exists an $n_1 < n$ with $T^k(n) = T^{\ell}(n_1)$. Simultaneously compute $T^{N-A}(n)$ for the remaining values of $n$.

\item For the remaining values of $n < 2^N$, compute $T^k(n)$ for $k > T^{N-A}(n)$ until this value falls below $n$.
\end{enumerate}
\caption{a high level overview of the algorithm to verify convergence for all $n < 2^N$.}
\end{algorithm}

The exclusion of $n \equiv 2, 4, 5, 8 \mod 9$ in Step (2) comes from the fact that in those cases $n = T^\ell(n_1)$ for some $n_1 < n$. To also detect $n$ for which $T^k(n) = T^{\ell}(n_1)$ for some $n_1 < n$ we add two more sieves to Steps (1) and (2). In theory these sieves could also be used in Step (3), but we have been unable to do so quickly enough to be worth it. See Section \ref{s:prelim} below for a description of all the sieves we use.

With this algorithm it should be possible to verify convergence for $n < 2^{77}$ with similar computer resources as what Barina used to verify convergence for $n < 2^{71}$. A more modest goal might be to verify convergence for all $n < 2^{75}$. We can then use the techniques in \cite{He23} to show that a nontrivial cycle, if one exists, must have length at least 2\,302\,268\,119\,908. That would be a 6.5-fold improvement on the result in their paper.

In our computer implementation we chose $A = 6$ for the CPU version and $A = 10$ for the GPU version, and $B = 24$ for both versions. The choice is based on benchmarking different values.

In \cite[Section 3.2]{Ba25} Barina states that using a $3^k$ sieve (with $k = 2$) is a time-critical operation. We implement the $3^2$ sieve, which we will refer to as the mod $9$ Preimage Sieve, by taking a bitwise \texttt{and} of one of $9$ precomputed bitvectors of length $2^A$ and a part of one of the bitvectors of length $2^B$ mentioned above. This reduces the number of mod $9$ calculations by a large factor and makes the mod $9$ Preimage Sieve very cheap to implement.

As we will explain below, our algorithm is equivalent to using a precomputed sieve of size $2^{N - A + B}$, where the precomputed bitvectors have size $2^B$, as well as the Mod 9 Preimage Sieve and two other sieves. By this we mean that the fraction of integers we actually need to consider is the same as the fraction one would get by using such a large precomputed sieve plus the other sieves.

If we set $N = 72$, $A = 10$, and $B = 24$, this means that our algorithm is equivalent to using a sieve of size $2^{86}$ plus the three other sieves. Moreover, Step (2) computes $T^{N-A}(n)$, which saves time when checking those numbers that survive the initial sieves. This compares to using a sieve of size $2^{34}$ from \cite[Table 4]{Ba25}. That is close to the limit of what is practical to store in memory; one advantage of our algorithm is that we avoid storing the large sieve.

One disadvantage of our algorithm is that it does not proceed linearly. It can, for example, be broken up into cases based on the last $k$ bits for a convenient value of $k$, but it cannot easily be adapted to check an interval $[K, L]$ of integers. This makes a direct comparison with Barina's work \cite{Ba25} difficult, as he breaks the computation up into intervals of length $2^{40}$.

\subsection{A computer implementation}
We have implemented our algorithm both for running on a CPU (in Rust) using native 128-bit integer types, and for running on a GPU using either CUDA or OpenCL using multiple 32-bit integers to represent a 128-bit integer value. We have made all the code available at \texttt{https://github.com/vigleik0/collatz}. The Rust program is about 340 lines of code, and does not contain any error checking. The CUDA program is about 700 lines, and the OpenCL program is about 900 lines. Both compute a hash, or checksum, and compares the result from running on the GPU to the result from running on the CPU for a random sample of inputs. Except for some differences in the boilerplate code required by CUDA and OpenGL the code is essentially identical, and both versions produce the same checksum.

To get the most out of a GPU, we basically have two choices. We could write a multi-threaded program with all the bells and whistles required to keep multiple ``GPU kernels'' in flight at the same time, or we could split the problem into multiple parts and run multiple copies of the program at the same time. We chose the latter option. Even with 8 copies of the program running concurrently, CPU usage is not very high, but we obtain close to 100\% utilization of the GPU.

Because the CUDA version is a little bit faster than the OpenCL version we will report on the running time for the CUDA version only.

Running on a GPU is, unsurprisingly, much faster than runnig on a CPU. On our computer we saw a difference of approximately a factor of 60. The running time appears to be significantly faster than other computer programs described in the literature. We estimate that it would take approximately 500 years of CPU time to verify the Collatz conjecture for all $n < 2^{72}$. Running on a GPU it would take somewhere in the range of 1-8 years, depending on the GPU. Here the 8 year estimate is based on benchmarking with a Nvidia RTX 3060 GPU and the 1 year estimate is based on the theoretical difference between a Nvidia RTX 3060 and a Nvidia RTX 5090. The calculation can of course be split up into cases to be processed in parallel.

We used the CPU version of our program to check all integers up to $2^{40}$ for convergence, using the appropriate sieves. It explicitly checked $756\,583\,624$ integers, or about 0.081\% of all natural numbers up to $2^{40}$. The program took 21 seconds, which translates to about $2^{35.6}$ integers per second, on a single AMD Zen 3 core running at 4.6GHz. We then ran the same program to check all integers up to $2^{50}$. This time it explicitly checked $434\,413\,363\,553$ integers, which is about 574 times as many. This time the calculation took 176 minutes, or about $2^{36.6}$ integers per second. In other words, it took about $2^9$ times as long to check $2^{10}$ times higher.

This is significantly faster than the 217s per $2^{40}$ numbers in \cite[Table 7]{Ba25}, especially when extrapolated to checking up to $2^N$ for some $N \geq 72$.

We then used the GPU version to check all integers up to $2^{50}$. This took 166 seconds on a Nvidia RTX 3060 with 3584 cores. This is about 60 times faster than the CPU version, or about $2^{42.6}$ integers per second. We then checked up to $2^{55}$. This run took 3636 seconds, or about $2^{43.2}$ integers per second. Finally we checked up to $2^{60}$. This took 24.4 hours, or about $2^{43.6}$ integers per second. This illustrates how the algorithm becomes better when checking a larger range. The GPU we ran the program on is obviously not state of the art: A Nvidia RTX 5090 has about 8 times the 32-bit integer arithmetic performance, and we assume it will speed up the calculation by a similar factor.

In \cite[Table 9]{Ba25}, Barina states that checking $2^{40}$ integers took 4.2s on a Nvidia RTX 2080 Ti. That's about $2^{37.9}$ integers per second. In other words, the new algorithm is about 50 times faster when checking up to $2^{60}$, and probably close to 100 times faster when checking up to $2^{72}$. (The Nvidia RTX 3060 and 2080 Ti have similar theoretical performance, so this should be a fair comparison.)

We estimate that with our algorithm we would have to explicitly check less than 0.02\% of all natural number up to $2^{72}$.

To check for correctness, we had our program output all starting values $n$ in the appropriate range for which $T^k(n)$ reached a large value. We compared this to the table of path records in \cite{Ba_table}, and verified that our program detects all the path records with a starting value below $2^{60}$.

To make a potential distributed computation possible, we have split the problem into 55, 1\,168, or 278\,699 separate cases of (at least conjecturally) similar running time.

\section{Some sieves} \label{s:prelim}
In this section we give some background material and discuss four different sieves. It is convenient to use the function $T$ instead of the function $C$, because we get to divide by $2$ exactly once each time.

\begin{lemma}[\cite{Si99}] \label{l:Tkfromlastbits}
The sequence of even and odd applications of $T$ in the $k$-fold composite $T^k(n)$ only depends on the last $k$ bits of the binary representation of $n$.
\end{lemma}

For example, because $3 = (11)_2$ and $11 = (1011)_2$ share the same last $3$ bits we can compute $T^3(3)$ and $T^3(11)$ using the same sequence of even and odd applications of $T$: odd, odd, even. For $3$ we get $3 \mapsto 5 \mapsto 8 \mapsto 4$ and for $11$ we get $11 \mapsto 17 \mapsto 26 \mapsto 13$.

\begin{defn}
We let $f_k(n)$ be the number of odd applications of $T$ to get to $T^k(n)$. Equivalently, $f_k(n)$ is the number of odd integers in $\{ n, T(n), \ldots, T^{k-1}(n) \}$.
\end{defn}

For example, $f_3(11) = 2$ since $2$ of the $3$ integers $\{ 11, T(11), T^2(11) \} = \{ 11, 17, 26 \}$ are odd. By Lemma \ref{l:Tkfromlastbits}, $f_k(n)$ depends only on the last $k$ bits of $n$.

\begin{lemma}[\cite{Si99}] \label{l:formulaforTk}
Suppose $n = n_0 + a 2^k$ for some $a \geq 0$ and that $T^k(n_0)$ has $f = f_k(n_0)$ odd applications. Then
\[
 T^k(n) = T^k(n_0) + 3^f a.
\]
\end{lemma}

For example, we can take $n_0 = 3$ and $n_1 = 3 + 1 \cdot 2^3 = 11$. Since $T^3(n_0)$ has $f = 2$ odd applications, $T^3(11) = T^3(3) + 3^2$. This has been used by various authors to implement a lookup-table for computing $T^k(n)$ in a single step by looking at the last $k$ bits of $n$, and we used this in our CPU implementation to do $k = 16$ steps at a time in Step (3) of Algorithm \ref{a:highlevel}. For our GPU implementation we found that it was better to avoid constantly looking up values in a table, and we instead used \cite[Algorithm 1]{Ba21}.

We immediately get the following:

\begin{cor}
Suppose $T^k(n_0)$ has $f = f_k(n_0)$ odd applications. If $T^k(n_0)$ is even then
\[
 T^{k+1}(n_0 + 2^k) = \frac{3(T^k(n_0) + 3^f) + 1}{2},
\]
and if $T^k(n_0)$ is odd then
\[
 T^{k+1}(n_0 + 2^k) = \frac{T^k(n_0) + 3^f}{2}.
\]
\end{cor}

This lets us compute both $T^{k+1}(n_0)$ and $T^{k+1}(n_0 + 2^k)$ from $T^k(n_0)$, so we can compute $T^M(n)$ for all $n < 2^M$ inductively by adding one bit at a time starting from the least significant bit. We keep track of the number $f$ of odd applications of $T$, and use a precomputed table of powers of $3$.

It is perhaps interesting to note that for any natural number $t < 2^M$, if we interpret the $k$'th bit in the binary representation of $t$ being $0$ as the $k$'th application of $T$ being even and the $k$'th bit being $1$ as the $k$'th application of $T$ being odd then there is exactly one natural number $n < 2^M$ for which repeated application of $T$ gives that pattern of even and odd. This defines a somewhat mysterious bijection $\bZ/2^M \to \bZ/2^M $. By letting $M$ go to infinity this gives a bijection from the $2$-adic integers to itself.

We now describe the different sieves we use.

\subsection{The Descent Sieve}
This sieve relies on the following result:

\begin{thm} \label{t:excludefromlastbits}
Suppose $n_1 = n_0 + a 2^k$ for some $a \geq 0$ and that $T^k(n_0) < n_0$. Then $T^k(n_1) < n_1$ as well.
\end{thm}

\begin{proof}
By Lemma \ref{l:Tkfromlastbits}, we use the same sequence of even and odd applications of $T$ to compute $T^k(n_0)$ and $T^k(n_1)$. If $m$ is even then $\frac{T(m)}{m} = \frac{1}{2}$ while if $m$ is odd then $\frac{T(m)}{m} = \frac{3}{2}\big( 1 + \frac{1}{3m} \big)$. Since $n_1 \geq n_0$ and both have the same sequence of even and odd applications of $T$, it follows by induction that $T^\ell(n_1) \geq T^\ell(n_0)$ for all $0 \leq \ell \leq k$, and hence also that $\frac{T^k(n_1)}{n_1} \leq \frac{T^k(n_0)}{n_0}$. The result follows.
\end{proof}

Given some $n_0 < 2^k$, the Descent Sieve works by checking if $T^k(n_0) < n_0$. If so, we exclude all $n$ given by adding more significant bits to $n_0$. This works because by induction we can assume that we have already verified convergence for $T^k(n)$ for any $n$ with the same last $k$ bits as $n_0$.

The basic idea behind this sieve is of course standard, and various authors have used it to precompute a table of mod $2^M$ values that need to be considered for some $M$. But the idea of inductively adding more significant bits appears to be new.

\subsection{The mod 9 Preimage Sieve}
This sieve relies on the following result:

\begin{lemma} \label{l:mod9sieve}
Suppose $n \equiv 2 \mod 3$. Then $n = T((2n-1)/3)$. Similarly, if $n \equiv 4 \mod 9$ then $n = T^3((8n-5)/9)$.
\end{lemma}

This means that if $n \equiv 2$, $4$, $5$ or $8 \mod 9$ then $n$ is on the path of either $(2n-1)/3$ or $(8n-5)/9$, so by induction it does not need to be considered.

This is another standard sieve. See for example \cite[Section 3.3]{Ba25}. There is a version of this sieve that considers the mod $3^k$ value of $n$ for any $k \geq 1$. The mod $27$ version does not exclude any more cases, and the mod $81$ version excludes just one more case, so we did not implement it.

\subsection{The Path-Merging Sieve}
We can also use Lemma \ref{l:mod9sieve} to prune a later part of the search tree.

\begin{lemma}
Suppose $m = T^k(n)$ and $m \equiv 2 \mod 3$. Then the path of $n$ joins the path of $(2m-1)/3$, so if $(2m-1)/3$ converges then so does $n$.
\end{lemma}

This sieve works by checking if $m = T^k(n) \equiv 2 \mod 3$ and $(2m-1)/3 < n$. If these conditions hold then by induction $n$ does not have to be checked.

Note that if $n = n_0 + a2^k$ and $k \geq 1$ it follows from Lemma \ref{l:formulaforTk} that $T^k(n) \equiv T^k(n_0) \mod 3$, so this sieve only depends on the last $k$ bits of $n$.

\begin{lemma} \label{l:mod3check}
Suppose $m = T^k(n)$ has at least one odd step. Then $m \equiv 2 \mod 3$ if and only if there is an even number of even steps after the last odd step.
\end{lemma}

For example, the path of $n = 15$ starts as $15 \mapsto 23 \mapsto 35 \mapsto 53 \mapsto 80 \mapsto 40 \mapsto 20$. Since there are $2$ even steps from the last odd step $53 \mapsto 80$ to $20$, we can conclude that $20 \equiv 2 \mod 3$ without doing the modular arithmetic calculation. Since $T(13) = 20$ we can then exclude $15$ from consideration. Because this example had a total of $6$ steps, it only depends on the last $6$ bits of $n = 15$. So this sieve excludes any $n \equiv 15 \mod 64$. For example, it also excludes $n = 79$, even though $T^6(79) = 101 > 79$ and the path of $79$ continues to increase for at least one more step.

There is a corresponding sieve that works whenever $m \equiv 4 \mod 9$, but it is less effective and we did not implement it.

Yet another version of this sieve works whenever $m \equiv 8 \mod 9$, because then $T^2((4m-5)/9) = m$ so it applies when $m \equiv 8 \mod 9$ and $(4m-5)/9 < n$. This sieve is also not very effective, so we did not implement it. More generally there is a version whenever $m \equiv -1 \mod 3^k$ for any $k$.

\subsection{The Odd-Even-Even Sieve}
This sieve relies on the following result:

\begin{lemma}[\protect{\cite[Lemma 9]{He23}}]
Suppose there are $\ell > 0$ odd steps in the Collatz sequence following $m$ and then at least $2$ even steps. Then the Collatz sequence of $m$ joins that of $(m-1)/2$.
\end{lemma}

The sieve works by checking if $m = T^k(n)$ is such that $(m-1)/2 < n$. When that happens, we look for a sequence of odd steps followed by $2$ even steps, and if such a sequence occurs we exclude $n$.

See Figure \ref{f:foursieves} for a summary of all four sieves. The $y$-axis is the value of $T^k(n)$ on a log scale, and the horizontal line is the starting value $n$.

\begin{figure}[htbp]
  \centering

  \begin{subfigure}{0.45\textwidth}
    \centering
    \includegraphics[width=\linewidth]{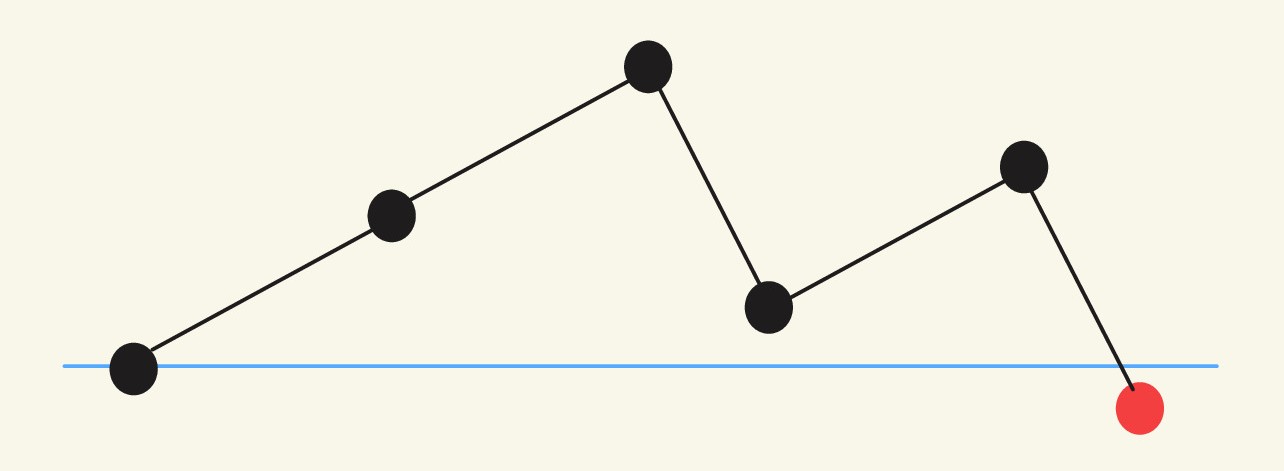}
    \caption{The Descent Sieve}
  \end{subfigure}
  \hfill
  \begin{subfigure}{0.45\textwidth}
    \centering
    \includegraphics[width=\linewidth]{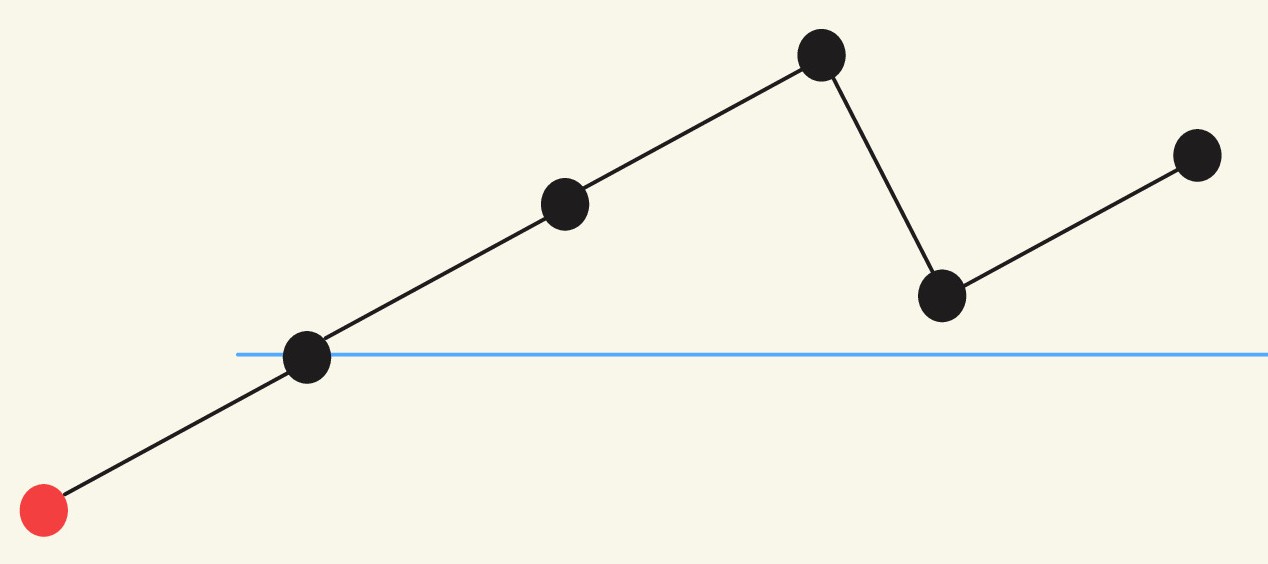}
    \caption{The mod 9 Preimage Sieve}
  \end{subfigure}

  \medskip

  \begin{subfigure}{0.45\textwidth}
    \centering
    \includegraphics[width=\linewidth]{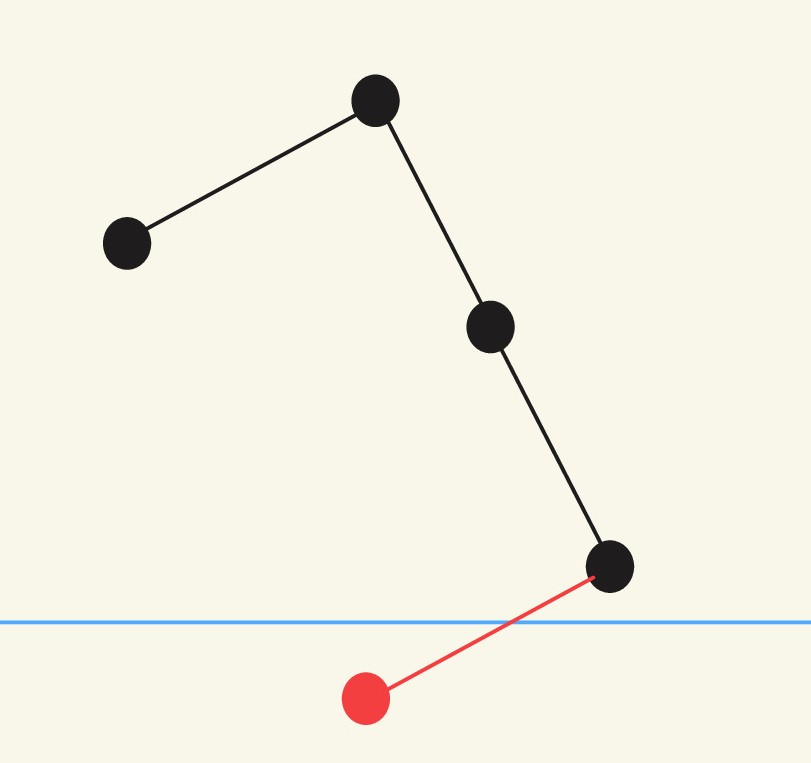}
    \caption{The Path-Merging Sieve}
  \end{subfigure}
  \hfill
  \begin{subfigure}{0.45\textwidth}
    \centering
    \includegraphics[width=\linewidth]{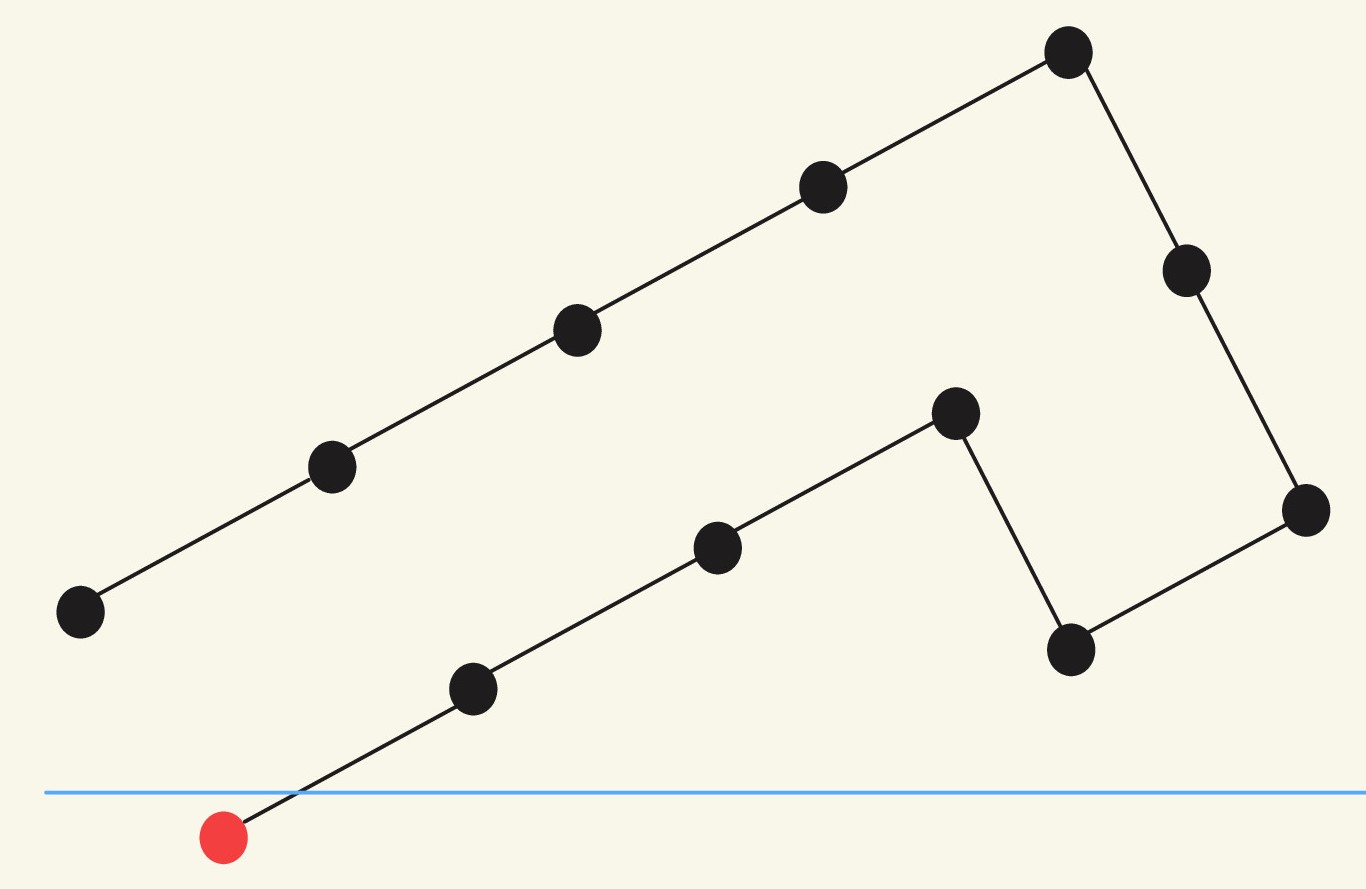}
    \caption{The Odd-Even-Even Sieve}
  \end{subfigure}

  \caption{The blue line indicates the starting value $n$, and the red dot indicates the value $n_1 < n$ whose path joins that of $n$.}
  \label{f:foursieves}
\end{figure}

\section{Step (1) of Algorithm \ref{a:highlevel}}
We now describe Step (1) of Algorithm \ref{a:highlevel} in more detail. This uses the Descent Sieve, the Path Merging Sieve, and the Odd-Even-Even sieve. We cannot use the mod 9 Preimage Sieve, as the mod 9 value of $n$ requires all the bits of $n$.

A limited version of the Descent Sieve has been used many times before. For example, in \cite{HIN17} the authors used this idea to precompute a sieve of size $2^k$ for $k$ as large as $37$, and in \cite{Ba25} the author used this idea to precompute a sieve of size $2^{34}$ (for use on a CPU) or size $2^{24}$ (for use on a GPU).

We use this algorithm for the least significant $M = N - A$ bits of $n$, and the method from the next section for the last $A$ bits, for some $A$ between $6$ and $10$. We can now describe Step (1) in Algorithm \ref{a:highlevel} in more detail, as Algorithm \ref{a:part1}.

\begin{algorithm}[H]
 \label{a:part1}
 \textbf{Require:} $n_0$, $m$, $k$ and $f$ are positive integers, $n_0 < 2^k$.
 \begin{enumerate}
  \item if $n_0$ is excluded by the Descent Sieve, the Path-Merging Sieve, or the Odd-Even-Even Sieve, stop;
  \item if $k = M$, output $(n_0, m, f)$ and then stop;
  \item if $m$ is even, call Algorithm \ref{a:part1} with $(n_0, m/2, k+1, f)$ and with $(n_0 + 2^k, (3(m+3^f)+1)/2, k+1, f+1)$;
  \item if $m$ is odd, call Algorithm \ref{a:part1} with $(n_0, (3m+1)/2, k+1, f+1)$ and $n_0 + 2^k, (m+3^f)/2, k+1, f)$.
 \end{enumerate}
 \caption{Step (1) in Algorithm \ref{a:highlevel}}
\end{algorithm}

We start Algorithm \ref{a:part1} with $(n_0, m, k, f) = (3, 8, 2, 2)$. We always have $m = T^k(n_0)$. The Descent Sieve is a simple comparison. For the Path-Merging Sieve we keep track of the number of even applications of $T$ after the last odd one and use Lemma \ref{l:mod3check} instead of an explicit mod 3 calculation. We also keep track of when the Odd-Even-Even sieve is in play.

\section{A rational approximation to $\frac{\ln(3)}{\ln(2)}$ and a precomputed bitvector}
In this section we describe the bitvectors used in Step (2) of Algorithm \ref{a:highlevel}. This consists of one bitvector for looking ahead another $B$ steps, for each value of $f$ coming out of Algorithm \ref{a:part1}, plus one bitvector for the mod $9$ Preimage Sieve. The bitvector for looking ahead another $B$ steps incorporates the Descent Sieve, the Path-Merging Sieve, and the Odd-Even-Even Sieve. Together with the bitvector for the mod $9$ Preimage Sieve that means we apply all four sieves in this step, with only a single bitwise \texttt{and} operation.

The key idea is to examine the last $B$ bits of $m = T^{N-A}(n)$ to check if $T^k(n)$ goes below $n$, or joins the path of some $n_1 < n$, for some $N - A < k \leq N - A + B$. These last bits of $m$ determine the next $B$ steps in the sequence of even and odd applications of $T$, starting from $m$.

Under mild conditions on $n$ and $k$ we only need to know $f = f_{N-A}(n)$ and the last $B$ bits of $m$ to determine if this happens.

For ease of computation we use a rational approximation to $\frac{\ln(3)}{\ln(2)}$. We have $\frac{485}{306} > \frac{\ln(3)}{\ln(2)}$, with $\frac{485}{306} - \frac{\ln{3}}{\ln{2}} = 0.000004819\ldots$. It follows that $3^{306}$ is slightly smaller than $2^{485}$, with $\frac{3^{306}}{2^{485}} = 0.9989\ldots$.

\begin{thm} \label{t:dip1}
Suppose $T^k(n)$ has $f = f_k(n)$ odd applications and $485 f \leq 306 k$. Suppose also that $T^m(n) \geq 99\,781$ for all $0 \leq m \leq k$. Then $T^k(n) < n$.
\end{thm}

\begin{proof}
Each odd application increases the number by a factor of at most $\frac{3}{2} \big( 1 + \frac{1}{3 \cdot 99\,783} \big)$, so $T^k(n)$ increases by a factor of at most $\frac{3^f}{2^k} \big( 1 + \frac{1}{3 \cdot 99\,783} \big)^f < 1$.
\end{proof}
The requirement of $T^m(n) \geq 99\,781$ for all $0 \leq m \leq k$ can be relaxed, but we leave the details to the reader since we do not need a better bound.

One can obviously prove a similar result for an even closer approximation to $\frac{\ln(3)}{\ln(2)}$, but the added benefit is very small.

We will take it for granted that the Collatz Conjecture has been verified for all $n \leq 99\,781$, so this means that for a given $n$ we can abort the search as soon as we have computed $T^k(n)$ and we find that $485 f_k(n) \leq 306 k$. In fact we do not even need to compute $T^k(n)$ if we can compute $f_k(n)$.

The same applies to the Path-Merging Sieve and the Odd-Even-Even sieve.

\subsection{Some bitvectors}
For each $m \in \{0,1,\ldots,2^B-1\}$ we compute the sequence of even and odd applications in $T^B(m)$.
\begin{defn}
Define a set $S_B'(m)$ of integers by taking the union over $k \in \{0,\ldots,B\}$ of $306k - 485f_k(m)$. Then define $\dip_B'(m) = \max_{s' \in S'}(s)$.

\end{defn}
The significance of $\dip_B'(m)$ is that for any $m$ for which $T^k(m) \geq 99\,781$ for all $0 \leq k \leq B$, $T^k(m) \leq 2^{-\dip_B'(m_0)/306} m$ for some $0 \leq k \leq B$, so this is useful for the Descent Sieve. We can also incorporate the Path-Merging Sieve and the Odd-Even-Even Sieve:

\begin{defn}
Define a set $S_B(m)$ of integers by taking the union over $k \in \{0,\ldots,B\}$ of $306k - 485f_k(m)$ and also the following numbers: If $T^k(m) \equiv 2 \mod 3$, include $306(k-1) - 485(f_k(m)-1)$. If $T^k(m)$ is followed by a positive number of odd steps and then 2 even steps, include $306(k+1) - 485f_k(m)$.

Then define $\dip_B(m) = \max_{s \in S}(s)$.
\end{defn}
It follows that for any $m$ for which $T^k(m)$ is large for all $0 \leq k \leq B$, either the path of $m$ dips below $2^{-\dip_B(m)/306} m$ or it joins the path of some $m_1 < 2^{-\dip_B(m_0)/306} m$. In either case, we can exclude $n$ if $2^{-\dip_B(m_0)/306} m < n$.

Now suppose we start with some $n$ and a calculation of $m = T^{N-A}(n)$. Suppose we also know $f = f_{N-A}(n)$. Then we can use $\dip_B(m)$ to check if one of our sieves apply:

\begin{thm} \label{t:dip2}
Let $m = T^{N-A}(n)$, let $f = f_{N-A}(n)$, and suppose $\dip_B(m) \geq 485f - 306(N-A)$. Suppose also that $T^m(n) \geq 99\,781$ for all $0 \leq m \leq N - A + B$. Then either $T^k(n) < n$ for some $N - A \leq k \leq N - A + B$ or the path of $n$ joins the path of some $n_1 < n$.
\end{thm}

Because $\dip_B(m) \geq 0$ for any $m$, Theorem \ref{t:dip2} contains Theorem \ref{t:dip1} as a special case.

There is a small chance that $T^k(n) < n$ even if $\dip_B(m) < 485f - 306(N-A)$, but this can only happen when $k$ is large and we will ignore it. (If this happens we end up checking $n$ in Step (3) of Algorithm \ref{a:highlevel} even though we do not have to.)

If the conditions in Theorem \ref{t:dip2} hold, the path of $n$ joins the path of some $n_1 < n$ and does not need to be considered.

For a fixed $f$, suppose $f_{N-A}(n) = f$. Then we must have $485f > 306(N-A)$, or $f \geq \lceil \frac{306(N-A)+1}{485} \rceil$. Here $\lceil - \rceil$ denotes the ceiling function, which rounds up to the nearest integer. We will consider a bitvector $BV_i$ of length $2^B$ corresponding to $f = \lceil \frac{306(N-A)+1}{485} \rceil + i$ for $i \geq 0$.

For example, suppose $N = 72$ and $A = 6$ so $N - 6 = 66$. Then we consider one bitvector $BV_i$ corresponding to $f = 42 + i$ for each $i \geq 0$.

In this example, if $n$ is so that $m = T^{66}(n)$ and $f_{66}(n) = 42$ it follows that if $\dip_B(m) \geq 485 \cdot 42 - 66 \cdot 306 = 174$ then $T^k(n) < n$ for some $66 < k \leq 66+B$ or the path of $n$ joins that of $n_1 < n$. We prepare the bitvector $BV_0$ accordingly, setting one bit for every $m \in \{0,1,\ldots,2^B-1\}$ with $\dip_B(m) < 174$.

See Figure \ref{f:fractionset} for how effective sieving with $BV_i$ is for the above example. We see that if $i$ is small then this sieve is very effective while $i$ is large it has almost no effect. In the CPU implementation we choose to limit ourselves to $i < 8$ and run the remaining cases without these sieves, while in the GPU implementation we chose to limit ourselves to $i < 16$ and run the very small number of remaining cases separately without this sieve on the CPU.

\begin{figure}
\begin{tabular}{ll}
Density of bits set for $BV_0$ & 0.0198687911 \\
Density of bits set for $BV_1$ & 0.0947877764 \\
Density of bits set for $BV_2$ & 0.2287583947 \\
Density of bits set for $BV_3$ & 0.4023408889 \\
Density of bits set for $BV_4$ & 0.5846776366 \\
Density of bits set for $BV_5$ & 0.7447901368 \\
Density of bits set for $BV_6$ & 0.8630392551 \\
Density of bits set for $BV_7$ & 0.9365259408 \\
\end{tabular}
\caption{The fraction of bits set in each bit vector $BV_i$ for $N = 72$, $A = 6$ and $B = 24$ for $0 \leq i \leq 7$.}
\label{f:fractionset}
\end{figure}

To speed up the computation, we do not order the bits in the bitvector for $BV_i$ linearly according to the last $24$ bits in $m$. Instead, if $BV_i$ corresponds to $f$ then bit number $a$ of $BV_i$ corresponds to $a 3^f \mod 2^B$. We do this so that the entries for $T^{N-A}(n_0)$ and $T^{N-A}(n_0 + 2^{N-A}) = T^{N-A}(n_0) + 3^f$ are consecutive. Hence the entries for $T^{N-A}(n_0 + a2^{N-A})$ for $a = 0, 1, \ldots, 2^A-1$ occupy $2^A$ adjacent entries in the bitvector.

\subsection{A bitvector for the mod $9$ Preimage Sieve} \label{ss:mod9sieve}
We refer once again to \cite[Section 3.2]{Ba25}. Because $T(2n+1) = 3n+2$, any $n \equiv 2 \mod 3$ does not have to be considered because it is in the path of a smaller number. Similarly, $T^3(8n+3) = 9n+4$, so any $n \equiv 4 \mod 9$ does not have to be considered. Together this removes $4/9$ of all the cases. As discussed in op.\ cit., using a mod $3^k$ Preimage Sieve for larger $k$ has limited additional benefit, so we stick with the mod $9$ Preimage Sieve.

Now we can complete our discussion of Step (2) of Algorithm \ref{a:highlevel}. Given some $f = f_{N-A}(n_0)$ for $n_0 < 2^{N-A}$, we use the precomputed bitvector $BV_i$ from above together with a precomputed bitvector for the mod $9$ Preimage Sieve to indicate for which $a$ the mod $9$ value of $n = n_0 + a 2^{N-A}$ is allowed, for $a=0,\ldots,2^A-1$.

Now we can use a single bitwise \texttt{and} operation of the two bitvectors to determine for which $a \in \{0, 1, \ldots, 2^A-1\}$ the number $n = n_0 + a2^{N-A}$ survives all the sieves.

\section{Step (3) of Algorithm \ref{a:highlevel}}
The output of Step (2) of Algorithm \ref{a:highlevel} is a set of natural numbers $n = n_0 + a2^{N-A}$ which need to be checked manually.

For our CPU implementation we used Lemma \ref{l:formulaforTk} to do $16$ steps at a time until $T^k(n)$ falls below $n$. The number of steps was dictated in part by the fact that the information needed to do $16$ steps can be stored in $32$ bits, while the information needed to do $17$ steps cannot (at least not easily). This sped up the program by a bit more than a factor of $2$, compared to using the same algorithm as the GPU version.

For our GPU implementation we used \cite[Algorithm 2]{Ba25}, which alternates between doing several odd applications of $T$ in a row and several even applications of $T$ in a row until $T^k(n)$ falls below $n$.

\subsection{Some empirical evidence}
In Figure \ref{f:casestocheck} we record how many integers need to be checked for each $N$ for $N = 35,\ldots,45$ using this method. Here $A = 6$ and $B = 24$, and we only used $8$ bitvectors $BV_0,\ldots,BV_7$.

\begin{figure}
\begin{tabular}{lrl}
N & cases to check & percentage to check \\
35 &      31\,109\,345 & 0.0905\% \\
36 &      56\,414\,851 & 0.0821\% \\
37 &     106\,551\,583 & 0.0775\% \\
38 &     205\,942\,265 & 0.0749\% \\
39 &     390\,055\,737 & 0.0710\% \\
40 &     756\,583\,624 & 0.0688\% \\
41 &  1\,408\,862\,059 & 0.0641\% \\
42 &  2\,658\,801\,769 & 0.0605\% \\
43 &  5\,143\,671\,444 & 0.0585\% \\
44 &  9\,392\,592\,637 & 0.0534\% \\
45 & 17\,807\,287\,297 & 0.0506\% \\
\end{tabular}
\caption{The percentage of all numbers to check for $N = 35,\ldots, 45$.}
\label{f:casestocheck}
\end{figure}

From this data it looks like considering one additional bit leads to about $1.9$ times as many cases. If the trend line holds, that means we have to consider approximately 0.0127\% of all numbers up to $2^N$ when $N = 72$ and $A = 6$, or 0.0156\% when $A = 10$.

We can compare this to \cite[Table 3]{HIN17}, where the authors conclude that using a size $2^{37}$ sieve leaves $0.70\%$ of numbers to be considered. Hence we get an improvement of about a factor of about 45 compared to their method.

\section{Some comments on the GPU implementation}
For the GPU implementation, we chose to do Step (1) of Algorithm \ref{a:highlevel} on the CPU and Steps (2) and (3) on the GPU. We prepare the necessary bitvectors on the CPU, and this takes about 3 seconds.

Doing Steps (2) and (3) on the GPU minimises the amount of data that has to be transferred from main memory to GPU memory. The bitvectors only have to be moved once, and then we process $2^{16}$ values of $n_0 < 2^{N-A}$ with the same $f$ at the same time, with each GPU thread responsible for one $n_0$. Because all the cases have the same $f$, this means each thread ends up being responsible for checking a similar number of cases $n = n_0 + a2^{N-A}$ and each thread runs for a similar amount of time. Typically this might be about 20 cases for the smallest possible $f$ and about 500 for a larger $f$.

We are of course free to choose the constants $A$ and $B$. Some benchmarking suggested $A = 10$ and $B = 24$ are good choices.

If all we care about is minimising the number of cases to be analysed in Step (3) then choosing $A = 0$ and the largest possible $B$ for which the bitvectors fit in memory is optimal. But it is a balancing act, for several reasons:
\begin{enumerate}
 \item Step (1) is relatively slow. It has branching, recursion, and arithmetic operations on 128-bit integers.
 \item Step (2) is very fast except for two bottlenecks: We need to compute a mod $9$ value, and we need to look up at least one bit in a big bitvector. With $B = 24$ most of the bitvector lives in the L3 cache (or the equivalent for the GPU) while if $B > 24$ memory access is slower. If we were to look up just one bit, for a single case, then the mod $9$ calculation and the latency from looking up that one bit in a bitvector would dominate. But if we look up $2^A$ bits we treat $2^A$ cases at the same time while only paying for the mod $9$ calculation and the latency once.
 \item The data generated in Step (1) has to be transferred from main memory to GPU memory for processing in Steps (2) and (3). If $A = 0$ then that is a lot of data to move around, while if $A = 10$ it is an almost trivial amount of data.
\end{enumerate}

Of course the choice of $A$ and $B$ that is optimal on one computer might not be optimal on a different computer.

\section{The analogue of the Collatz conjecture for negative integers}
The Collatz function $C$ and the modified Collatz function $T$ makes sense for any integer $n$. As is well known, if we extend the Collatz function to all integers there are a few more cycles:
\begin{eqnarray*}
 & & (1, 2, 1) \\
 & & (0) \\
 & & (-1) \\
 & & (-5, -7, -10, -5) \\
 & & (-17, -25, -37, -55, -82, -41, -61, -91, -136, -68, -34, -17)
\end{eqnarray*}
The Collatz conjecture, extended to all integers, says that every integer $n$ eventually hits one of these cycles.

The algorithm in this paper applies, without any substantial changes, to negative integers as well. We can implement this using the function
\[
 T'(n) = \begin{cases} n/2 \quad & \textnormal{if $n$ is even} \\
                       (3n-1)/2 \quad & \textnormal{if $n$ is odd}
         \end{cases}
\]
for positive $n$, to avoid having to deal with negative numbers in software. Theorem \ref{t:excludefromlastbits} does not hold, but we can instead check if $306k - 485f_k(n) \geq 0$. With this formulation of the Descent Sieve, all four sieves described in Section \ref{s:prelim} apply, and the running time of our program using $T'$ in place of $T$ is almost identical. In Appendix \ref{a:pathrecordsnegative} we record the path records we found for $T'$. We see that the data supports the analogue of the conjecture predicted by the independent random walk process in \cite[Theorem 2.3]{LaWe92} that $\lim \sup \frac{\log(t(n))}{\log(n)} = 2$ for $T'$, and the data for $T'$ does not look substantially different from the data for $T$.

\section{Conclusion}
We claim that the algorithm presented in this paper is about 50-100 times faster than other algorithms discussed in the literature, both because we end up with fewer cases to consider and because some of the work can be done very effectively using bitwise operations.

We hope to report on the verification of convergence for $n < 2^N$ for some $N \geq 72$ in future work.

\appendix
\section{Path records for negative numbers} \label{a:pathrecordsnegative}
We consider the function $T'$ on positive integers instead of $T$ on negative integers. Define $t'(n)$ to be highest value occuring in the sequence starting at $n$. We say $n > 0$ is a path record if $t'(n) > t'(n_1)$ for all $n_1 < n$. See Figure \ref{f:pathrecords}. This should be compared to the table in \cite{Ba_table}.

\clearpage
\begin{figure}[p]
\vspace*{-80pt}
\hspace*{-40pt}
\resizebox{!}{0.60\textheight}{
\begin{tabular}{l|r|r|l|l}
number & $n$ & $t'(n)$ & $\log(t'(n))/\log(n)$ & comments \\
\hline
 1 &                       1 &                                                      1 &  & cycle \\
 2 &                       2 &                                                      2 & 1.0000 & even \\
 3 &                       3 &                                                      4 & 1.2619 & \\
 4 &                       5 &                                                     10 & 1.4307 & cycle \\
 5 &                       9 &                                                     28 & 1.5000 & \\
 6 &                      17 &                                                    136 & 1.7340 & cycle \\
 7 &                      33 &                                                    244 & 1.5722 & \\
 8 &                      65 &                                                    820 & 1.6072 & \\
 9 &                     129 &                                                  2 188 & 1.5825 & \\
10 &                     153 &                                                 16 606 & 1.9324 & log ratio \\
11 &                     321 &                                                 66 430 & 1.9239 & \\
12 &                   1 601 &                                                131 356 & 1.5973 & \\
13 &                   1 889 &                                                413 344 & 1.7143 & \\
14 &                   3 393 &                                                417 718 & 1.5921 & \\
15 &                   4 097 &                                                957 664 & 1.6557 & \\
16 &                   6 929 &                                              1 439 776 & 1.6034 & \\
17 &                   8 193 &                                              1 594 324 & 1.5849 & \\
18 &                  10 497 &                                              2 908 468 & 1.6075 & \\
19 &                  11 025 &                                             40 219 750 & 1.8812 & \\
20 &                  18 273 &                                             44 442 028 & 1.7945 & \\
21 &                  28 161 &                                            195 046 228 & 1.8631 & \\
22 &                  74 585 &                                            477 250 624 & 1.7811 \\
23 &                  85 265 &                                            510 919 012 & 1.7661 \\
24 &                 149 345 &                                          4 837 921 750 & 1.8717 \\
25 &                 558 341 &                                         39 156 432 022 & 1.8432 \\
26 &                 839 429 &                                         39 246 157 990 & 1.7883 \\
27 &               1 022 105 &                                         45 360 267 382 & 1.7733 \\
28 &               1 467 393 &                                      3 293 075 932 912 & 2.0299 & log ratio \\
29 &               7 932 689 &                                      7 033 004 986 294 & 1.8621 \\
30 &               8 612 097 &                                     15 270 716 514 700 & 1.9010 \\
31 &              23 911 397 &                                     39 704 218 231 240 & 1.8430 \\
32 &              58 882 625 &                                    127 143 512 668 792 & 1.8152 \\
33 &              75 567 105 &                                  1 101 396 273 700 744 & 1.9093 & \\
34 &             293 056 017 &                                  2 999 880 103 468 384 & 1.8279 \\
35 &             299 480 577 &                                  7 168 907 674 555 372 & 1.8705 \\
36 &             344 371 457 &                                  8 469 967 658 156 470 & 1.8657 \\
37 &             677 585 217 &                                 85 520 198 863 910 422 & 1.9174 &  \\
38 &             788 620 517 &                                111 976 109 464 769 974 & 1.9163 \\
39 &           1 951 587 609 &                                729 002 383 530 173 326 & 1.9227 \\
40 &           2 672 464 025 &                             18 462 714 214 850 651 206 & 2.0438 & log ratio \\
41 &          15 958 182 629 &                            174 217 613 946 575 461 336 & 1.9838 \\
42 &          52 002 133 905 &                            708 842 423 362 106 898 604 & 1.9457\\
43 &         187 559 691 777 &                         28 019 930 157 696 441 632 608 & 1.9912 & \\
44 &         213 121 397 657 &                        261 160 802 435 320 822 179 964 & 2.0671 & log ratio \\
45 &       2 102 553 018 369 &                      2 326 308 864 144 659 904 362 302 & 1.9774 \\
46 &       4 092 376 481 793 &                      3 435 069 575 515 199 820 916 438 & 1.9454 \\
47 &       9 425 800 771 749 &                      8 466 499 925 136 151 018 385 248 & 1.9213 \\
48 &      10 995 522 036 993 &                     10 011 211 698 317 452 027 896 052 & 1.9170 \\
49 &      17 364 775 952 385 &                    121 443 575 752 945 981 388 885 320 & 1.9702 \\
50 &      37 673 395 808 865 &                    360 951 744 282 085 208 020 927 096 & 1.9562 \\
51 &      46 787 997 137 637 &                    906 895 196 063 261 171 526 071 584 & 1.9720 \\
52 &     149 782 712 025 089 &                  1 210 848 622 594 996 854 542 206 672 & 1.9106 \\
53 &     207 834 647 861 121 &                  1 243 146 285 315 831 843 449 444 584 & 1.8924 \\
54 &     245 192 561 861 633 &                 21 142 881 571 964 420 917 882 086 904 & 1.9685 \\
55 &     787 788 664 859 621 &                 41 274 920 521 729 277 074 671 487 084 & 1.9210 \\
56 &   1 158 406 214 795 493 &                 44 906 946 035 645 153 713 404 631 060 & 1.9021 \\
57 &   1 277 092 890 862 593 &                 90 432 155 160 573 081 675 193 686 196 & 1.9168 \\
58 &   1 890 156 036 203 417 &                154 710 878 312 175 791 134 940 888 500 & 1.9107 \\
59 &   2 618 767 040 913 665 &                928 077 024 906 035 572 854 112 157 278 & 1.9437 \\
60 &   4 884 898 276 413 393 &              2 366 708 823 221 238 093 371 274 372 556 & 1.9360 \\
61 &  11 397 570 803 793 921 &            145 412 571 545 812 686 532 103 947 205 632 & 2.0031 \\
62 &  22 375 364 831 208 729 &            197 795 804 242 092 643 270 085 071 859 668 & 1.9753 \\
63 &  90 370 152 179 192 729 &            830 272 236 784 374 542 016 248 397 039 502 & 1.9414 \\
64 & 156 427 084 727 752 601 &          1 620 193 332 937 735 208 990 463 312 696 940 & 1.9314 \\
65 & 340 848 634 085 148 225 &          4 398 752 114 240 261 198 345 796 424 029 598 & 1.9189 \\
66 & 460 533 332 579 759 073 &          4 579 947 937 528 522 259 447 176 209 253 684 & 1.9057 \\
67 & 826 013 836 782 232 577 &          6 403 193 270 344 035 112 633 081 355 482 246 & 1.8868 \\
68 & 929 160 718 932 705 509 &         19 469 464 240 235 887 857 456 644 898 658 276 & 1.9083 \\
\end{tabular}
}
\caption{Path records for $T'$.}
\label{f:pathrecords}
\end{figure}



\end{document}